\documentclass[12pt,reqno]{amsart}
\newcommand{\be}{\begin{eqnarray}}
\newcommand{\ee}{\end{eqnarray}}

\newcommand{\card}{\mbox{\rm card}}

\newcommand{\viert}{\frac{1}{4}}
\newcommand{\eps}{\epsilon}
\newcommand{\del}{\delta}

\newcommand{\Z}{{\mathbb Z}}

\newcommand{\F}{{\mathcal F}}

\newcommand{\Seq}{{\mathcal S}}
\newcommand{\HH}{{\mathcal H}}
\newcommand{\A}{{\mathcal A}}
\newcommand{\Gr}{{\mathcal G}}
\newcommand{\tor}{\mathbb T}
\renewcommand{\P}{{\mathbb P}}

\newtheorem{theorem}{Theorem}[section]
\newtheorem{lemma}[theorem]{Lemma}
\newtheorem{cor}[theorem]{Corollary}

\theoremstyle{definition}

\theoremstyle{remark}

\numberwithin{equation}{section}

\begin{document}

\title[Two Erd\H{o}s problems]{Two Erd\H{o}s problems on lacunary sequences:
 \mbox{Chromatic~number~and~Diophantine~approximation}}

\author{Yuval Peres}

\address{Microsoft Research, Redmond, WA
  and University
 of California, Berkeley, CA}
\email{peres@microsoft.com}

\author{Wilhelm Schlag}\thanks{Research of Peres was
partially supported by NSF grant  DMS-0605166.
 Research of Schlag was partially supported by NSF
grant DMS-0617854.}
%xxx
\address{Department of Mathematics, University of Chicago.}

\email{schlag@math.uchicago.edu}
%xxx
 % Diophantine approximation: Density, gaps, topology


\begin{abstract}
Let $\{n_k\}$ be an increasing lacunary sequence, {\em i.e.},
 $n_{k+1}/n_k>1+\epsilon$ for some $\epsilon>0$. In 1987,
P. Erd\H{o}s asked for the chromatic number $\chi(G)$ of a graph $G$ with vertex set $\Z$,
where two integers  $x,y \in \Z$ are connected by an edge
iff their difference $|x-y|$ is in the sequence $\{n_k\}$.
Y.~Katznelson found  a connection to a Diophantine approximation problem (also due
to Erd\H{o}s): the existence of $\theta \in (0,1)$  such that
all the multiples $n_j \theta$ are at least distance $\del(\theta)>0$
from $\Z$. Katznelson showed that
%$\del>c \epsilon^2 |\log \eps|^{-1}$ for an appropriate $\theta$,
$\chi(G) \le C\epsilon^{-2}|\log \eps|$.
We apply the Lov\'asz local lemma to establish that
$\del(\theta)>c \epsilon |\log \eps|^{-1}$ for some $\theta$,
which implies that $\chi(G)< C\epsilon^{-1} |\log \eps|$.
This is sharp up to the logarithmic factor.
\end{abstract}

\maketitle

\section{{\bf Introduction}}

\noindent
The chromatic number~$\chi(\Gr)$ of a graph $\Gr$
is the minimal number of colors that can be assigned to the vertices of $\Gr$
 so that no edge connects two vertices of the same color.
In 1987, Paul Erd\H{o}s posed the following
problem\footnote{according to Y. Katznelson~\cite{Kat}.}:

\vspace{.3cm} \it
\noindent {\bf Problem A:} Let $\eps>0$ be fixed and suppose $\Seq=\{n_j\}_{j=1}^\infty$ is a sequence of
positive integers such that $n_{j+1}>(1+\eps) n_j$ for all~$j \ge 1$. Define a graph~$\Gr=\Gr(\Seq)$ with
vertex set~$\Z$ (the integers)
 by letting the pair $(n,m)$ be an edge iff $|n-m|\in\Seq$.
Is the chromatic number~$\chi(\Gr)$ finite?

\vspace{.3cm} \rm

\noindent There is a related Diophantine approximation problem, posed earlier by Erd\H{o}s~\cite{E1}:
\vspace{.3cm}

\noindent {\bf Problem B:} \it
Let $\eps>0$ and $\Seq$ be as above.
Is there a number $\theta\in(0,1)$ so that the sequence
$\{n_j \theta\}_{j=1}^\infty$ is not dense modulo $ 1$?
\rm

\vspace{.4cm}
\noindent The relation between Problem~A and Problem~B was discovered by Katznelson~\cite{Kat}: Let $\del>0$ and
$\theta\in(0,1)$ be such that $\inf_j\|\theta n_j\|>\del$, where $\|\cdot\|$ denotes the distance to the
closest integer. Partition the circle $\tor=[0,1)$ into $k=\lceil\del^{-1}\rceil$
disjoint intervals $I_1,\ldots, I_k$ of length
$\frac{1}{k}\le\del$. Let~$\Gr$ be the graph from Problem~A and color the vertex~$n\in\Z$ with color~$j$
iff $n \theta \in I_j \pmod{1}$. Clearly, any two vertices connected by an edge must have different
 colors. Therefore, $\chi(\Gr) \le k =\lceil\del^{-1}\rceil$.
 From this, one can easily deduce that $\chi(\Gr)<\infty$ for any lacunary sequence with ratio at least $1+\epsilon$
 by partitioning  into several subsequences
 (see the end of the introduction); however, the bound obtained this way grows exponentially in $\eps^{-1}$.

 Problem~B was solved by de Mathan~\cite{Ma} and Pollington~\cite{Poll} and their proofs provide
 bounds on~$\chi(\Gr)$ that grow polynomially in~$\eps^{-1}$. More precisely, they show
that there exists~$\theta\in(0,1)$ such that
\[ \inf_{j\ge1}\|\theta n_j\|>c \eps ^{4}|\log\eps|^{-1}\]
where $c>0$ is some constant. This bound was improved by Katznelson~\cite{Kat},
who showed that there exists a $\theta$ such that
\be \label{katz} \inf_{j\ge 1}\|\theta n_j\|>c \eps^{2}|\log \eps|^{-1}. \ee
Akhunzhanov and  Moshchevitin~\cite{AM}  removed
the logarithmic factor on the right hand side of \eqref{katz},  see also Dubickas~\cite{Du}.

We can now state the main result of this note.

\begin{theorem} \label{thm:main}
Suppose $\Seq=\{n_j\}$ satisfies $n_{j+1}/n_j \ge 1+\eps$, where $0<\eps<1/4$. 
Then there exists $\theta \in (0,1)$ such that
\be\label{mlogm}
\inf_{j\ge 1}\|\theta n_j\|>c \eps|\log \eps|^{-1}\,,
\ee
where $c>0$ is a universal constant. Therefore, the graph $\Gr=\Gr(\Seq)$ described in Problem A
satisfies $\chi(\Gr) \le 1+c^{-1} \eps^{-1} |\log \eps|$.
\end{theorem}
Up to the $|\log \eps|^{-1}$ factor, (\ref{mlogm}) cannot be improved. Indeed, let $n_j=j$
for~$j=1,2,\ldots,\lfloor \eps^{-1}\rfloor$ and continue this as a lacunary sequence with
ratio~$1+\eps$.  It is clear that  $\chi(\Gr)>\lfloor{\eps}^{-1}\rfloor$ in this case,
so that the power of~$\eps$
in~\eqref{mlogm} cannot be decreased.

In order to prove Theorem \ref{thm:main}, we use the Lov\'asz local lemma from probabilistic combinatorics,
see \cite{EL} or \cite[Chap.\ 5]{AES}. Loosely speaking, given events $A_1,A_2,\ldots$ in a
probability space, this lemma bounds
$ \P\Bigl(\bigcap_{j=1}^N A_j^c\Bigr) $
from below if the events $A_j$ have small probability and each $A_i$ is almost independent of most of the others.
See the following section for a precise statement.
Theorem \ref{thm:main} is established in Section~\ref{rotation}.
%In the application to Problem~B, one lets
%\[ A_j=\{\theta\in\tor\::\: \|\theta n_j\|>\delta\}.\] If $\delta (1+\eps)^K>1$, then $A_j$ is
%basically independent of all events $A_i$ with~$i<j-K$. The local lemma determines how
%small~$\delta$ can be taken relative to~$K$.

Finally, we recall how the aforementioned relation between problems~A and~B
yields an easy proof that $\chi(\Gr)<\infty$ for any $\eps>0$.
See \cite{Kat}, \cite[Chap.\ 5]{Wei} and  \cite{RTV}  for variants of this argument.
 First suppose that~$n_{j+1}/n_j>4$ for all $j$.
In this case,
\be\label{durch} && \bigcap_{j=1}^\infty \Bigl\{\theta\in\tor\::\:\|\theta n_j\|>\viert\Bigr\}\not=\emptyset.\ee
Indeed, fix some $j$ and notice that the set $\{\theta\in\tor\::\:\|\theta n_j\|>\viert\}$ is the union of the
middle halves of intervals~$[\frac{\ell}{n_j},\frac{\ell+1}{n_j}]$ where $\ell=0,1,\ldots,n_j-1$.
Since $n_{j+1}>4n_j$, each such middle half contains an entire interval of the form
$[\frac{\ell'}{n_{j+1}},\frac{\ell'+1}{n_{j+1}}]$. Iterating this yields a sequence
of nested intervals and establishes (\ref{durch}).

  Now suppose just that $n_{j+1}/n_j>1+\eps>1$.
Pick~$K=\lceil 2 \eps^{-1}\rceil$ so that $(1+\eps)^K>4$.
Divide the given sequence~$\Seq$ into~$K$ subsequences
$\{n_{Kj+r}\}_{j=0}^\infty$, with $r=1,2,\ldots,K$. Applying~\eqref{durch} to
each such subsequence
yields $(\theta_1,\ldots,\theta_K)\in\tor^K$ so that
\[ \inf_{j \ge 0}\|n_{Kj+r}\:\theta_r\|\ge\viert \text{\ \ for all }r=1,2,\ldots,K.\]
Coloring each integer $m$ according to which quarter of the unit interval
$m\theta_r$ falls into
for $1 \le r \le K$, shows that $\chi(\Gr)\le 4^K$.
Observe that as $\eps \to 0$, this bound grows exponentially in~$\frac{1}{\eps}$.

\section{{\bf A one-sided version of the local lemma}}

The following lemma is the variant of the Lov\'asz local lemma~\cite{EL} that
 we apply to Problem~B above.
Since it is not exactly stated in this form (neither in terms of the
 hypotheses nor the conclusion)
in \cite{EL} or ~\cite{AES}, we provide a proof for the reader's convenience.
We stress, however, that it is a simple adaptation of the argument
 given in chapter~5 of~\cite{AES}.

\begin{lemma}\label{loclem} Let $\{A_j\}_{j=1}^N$ be events in a
probability space $(\Omega, {\F}, \P)$
and let $\{x_j\}_{j=1}^N$ be a sequence of numbers in~$(0,1)$. Assume that for
every~$i\le N$,
there is an integer~$0\le m(i)<i$ so that
\be\label{cond} \P\Bigl(A_i\,\Big|\bigcap_{j< m(i)}A_j^c\Bigr)
&\le& x_i\prod_{j=m(i)}^{i-1}(1-x_j).\ee
Then for any integer  $n \in [1, N]$, we have
%\be \P\Bigl(\bigcap_{i=1}^{k+n} A_i^c\Bigr) &\ge& 
\be \P\Bigl(\bigcap_{i=1}^{n} A_i^c\Bigr) &\ge& 
\prod_{\ell=1}^{n}(1-x_\ell)\label{lower}.\ee 
\end{lemma}
\begin{proof} Denote $B_1=\Omega$ and
$B_\ell=A_1^c\cap\ldots\cap A_{\ell-1}^c$ for $\ell > 1$.
 We claim that for each $\ell\ge 1$,
\be\label{claim} \P(A_\ell\,|B_\ell) &\le& x_\ell.\ee
Since
%\[\P\Bigl(\bigcap_{i=1}^{k+n} A_i^c\Bigr) = \P\Bigl(\bigcap_{i=1}^{k}A_i^c\Bigr) \prod_{\ell=k+1}^{k+n}
%\Bigl(1-\P(A_{\ell}\, | B_{\ell})\Bigr),\]
\[\P\Bigl(\bigcap_{i=1}^{n} A_i^c\Bigr) = \prod_{\ell=1}^{n}
\Bigl(1-\P(A_{\ell}\, | B_{\ell})\Bigr),\]

\eqref{claim} implies~\eqref{lower}.
The claim  \eqref{claim} is verified inductively:

\be \P\Bigl(A_{\ell}\,\Big|B_\ell\Bigr) =
    \frac{\P\Bigl(A_\ell\cap B_\ell \, \Big|B_{m(\ell)}\Bigr)}
         {\P\Bigl( B_\ell \, \Big|B_{m(\ell)}\Bigr)}   \le
         \frac{\P\Bigl(A_\ell \, \Big| B_{m(\ell)}\Bigr)}
              {\P\Bigl( B_\ell \, \Big| B_{m(\ell)}\Bigr)} \,.
\label{rat}
\ee
The denominator in the rightmost fraction can be written as a product
\be
 \P\Bigl( B_\ell \, \Big| B_{m(\ell)}\Bigr)=\prod_{j=m(\ell)}^{\ell-1} \Bigl(1-\P(A_j \,| B_j)\Bigr) \, .
\ee
By the inductive hypothesis, this product is at least $\prod_{j=m(\ell)}^{\ell-1}(1-x_j)$,
whereas the numerator in the right-hand side of ~\eqref{rat} is at most
$x_\ell\prod_{j=m(\ell)}^{\ell-1}(1-x_j)$ by~\eqref{cond}. This finishes the proof.
\end{proof}

\section{{\bf Rotation orbits sampled along a lacunary sequence}} \label{rotation}
In this section we present our quantitative result on problem B,
which extends Theorem~\ref{thm:main} and is also applicable to unions of lacunary sequences. Observe that if  $\Seq=\{n_j\}$
is lacunary with ratio $1+\eps$, then it satisfies the hypothesis of the next
theorem with $M=\lceil \eps^{-1} \rceil$. More generally, if $\Seq_1,\ldots,\Seq_\ell$
are lacunary with ratios $1+\eps_1,\ldots,1+\eps_\ell$ respectively,
then their union satisfies the hypothesis with
 $M=\sum_{i=1}^\ell \lceil \eps_i^{-1} \rceil$. We shall assume that $M\ge 4$.

\begin{theorem} \label{thm:dio}
Suppose $\Seq=\{n_j\}$ satisfies $n_{j+M}>2n_j$ for all $j$. Define
\be\label{Ej} E_j &=& \biggl\{\theta\in\tor\::\: \|n_j\theta\|<\frac{c_0}{M \log_2 M}\biggr\}\ee
for $j \ge 1$. If $\;240\,c_0\le 1$, then
\be\label{inter} \bigcap_{j=1}^\infty E_j^c &\not=& \emptyset.\ee
\end{theorem}
\begin{proof} Set
\be\label{deldef} \delta &=& \frac{c_0}{M\log_2 M}.\ee
For each $j=1,2,\ldots$ define an integer $\ell_j$ by
\be\label{scale} && 2^{-\ell_j-1} < \frac{2\delta}{n_j}\le 2^{-\ell_j}.\ee
 Let $A_j$ be the union of all the open dyadic intervals of size $2^{-\ell_j}$ that intersect $E_j$.
 Observe that
$E_j$ is the union of $n_j$ intervals of length~$\frac{2\delta}{n_j}$,  and each one of them is covered by at most two dyadic intervals of length~$2^{-\ell_j}$. Therefore,
\be\label{Aj} && \P(A_j) \le 2\cdot 2^{-\ell_j}n_j\leq 8\delta.\ee
  where $\P$ is Lebesgue measure on $[0,1]$.
Define
\be\label{hdef} h &=& \lceil C_1 \log_2 M\rceil M ,\ee
where $C_1\ge 5$ is a constant to be determined. Our goal is to apply Lemma~\ref{loclem} with $m(i)=i-h$
 and $x_i=x= h^{-1}$ for all $i\in\Z^+$. To verify~\eqref{cond}, fix some $i>h$. Then
\[ \bigcap_{j<i-h} A_j^c = \bigcup_s I_s\]
with dyadic intervals~$I_s$ of length $|I_s|=2^{-\ell_{i-h-1}}$. Hence
\be && \P\Bigl(A_i\cap\bigcap_{j<i-h} A_j^c\Bigr) = \sum_s\P(A_i\cap I_s) \le \sum_s(1+|I_s|n_i)2^{1-\ell_i}\nonumber\\
&& \le \P\Bigl(\bigcap_{j<i-h} A_j^c\Bigr)\Bigl[2\,n_i2^{-\ell_i}+2^{1+\ell_{i-h-1}-\ell_i}\Bigr]\nonumber\\
&&  \le \P\Bigl(\bigcap_{j<i-h} A_j^c\Bigr)\Bigl[8\delta+4\frac{n_{i-h-1}}{n_i}\Bigr].\label{pre}
\ee
To pass to \eqref{pre}, one uses \eqref{scale}.
By~\eqref{hdef},
\be\label{qh1} \frac{n_{i-h-1}}{n_i}\le 2^{-C_1\log_2 M}=M^{-C_1}.\ee
Inserting this bound into~\eqref{pre} yields
\be && \label{12delta} \P\Bigl(A_i\:\Big|\:\bigcap_{j<i-h} A_j^c\Bigr) \le 12\delta,\text{\ \ provided that}\\
 && \frac{c_0M^{C_1}}{M\log_2 M} \ge 1 \text{,\ \ which is the same as\ \ }  M^{-C_1}\le\delta. \label{bed1}
\ee
By \eqref{Aj}, the estimate \eqref{12delta} holds also if $i\le h$.
In order to satisfy \eqref{cond}, we need to ensure that
\be \label{mystery} 12\delta \le x(1-x)^h. \ee
Since $x= h^{-1}\le 1/16$, we have $(1-x)^h\ge 1/3$. Thus \eqref{mystery} will be satisfied if 
\be\label{range} 36\delta \le x=h^{-1}.\ee
By \eqref{deldef} and~\eqref{hdef}, $h \delta=\lceil C_1 \log_2 M\rceil \, c_0/\log_2 M \le \frac{10}{9} C_1c_0$,
 since $M \ge 4$ and $C_1 \ge 5$. Therefore,  \eqref{range} will hold if
\be\label{bed2} 40 C_1 c_0\le 1 \,.\ee
Take $C_1=6$ and $c_0=\frac{1}{240}$; then \eqref{bed2} and \eqref{bed1} are both satisfied.
 By Lemma~\ref{loclem},
\be\label{exp} \P\Bigl(\bigcap_{i=1}^n A_i^c\Bigr) &\ge& \Bigl(1-x\Bigr)^n \,.\ee
Since each of the $A_j^c$ is compact, and $A_j^c \subset E_j^c$, the intersection $\cap_{j=1}^\infty E_j^c$ is nonempty, as claimed.
\end{proof}

\section{{\bf Intersective sets}}

Let $\Seq=\{n_j\}$ with $n_{j+1}\ge(1+\eps)n_j$ be a lacunary sequence of positive integers with ratio $1+\eps>1$. Denote
$M=\lceil \eps^{-1} \rceil$.
We have shown that there is a coloring of the graph~$\Gr$ with at most $CM\log M$ colors.
 Let $\A_{\max}$ be a set of integers of the same color  with upper density
\[ D^*(\A_{max})=\limsup_{N\to\infty}\frac{\card(\A_{\max}\cap [-N,N])}{2N+1} > \frac{c}{M\log M},\]
where $c$ is a constant. By the definition of coloring,
\[ (\A_{\max}-\A_{\max})\cap \Seq=\emptyset.\]
A set $\HH$ is called {\em intersective} if
\[ (\A-\A)\cap\HH \not= \emptyset\]
for any $\A\subset\Z$ with $D^*(\A)>0$. Intersective sets are precisely the Poincar\'e
sets considered by F\"urstenberg~\cite{Furb}, see \cite{BM} and \cite{per}. Generally speaking,
it is not a simple matter to decide whether a given set is intersective or not.
F\"urstenberg~\cite{Fur} and
 S\'{a}rk\"{o}zy~\cite{Sza}  showed that the squares (and more generally, the set $\{P(n)\}_{n \in \Z}$
  where $P$ is a polynomial over $\Z$ that vanishes at some integer) are intersective. In particular,
intersective sets can have zero density. On the other hand,
 the previous discussion shows that any
finite union of lacunary sequences is not intersective.
There are some related concepts of intersectivity
which we briefly recall; for a nice introduction to this subject see chapter~2 in
Montgomery's book~\cite{Mont}. A set of integers $\HH$ is called a {\em van der Corput} set,
if for any sequence~$\{x_j\}_j$ of numbers with the property that $\{x_{j+h}-x_j\}_j$ is uniformly
distributed modulo $1$ for all~$h\in\HH$, one has that $\{x_j\}_j$ is uniformly distributed
modulo $1$. It was shown by  Kamae and Mendes--France~\cite{KM} that any van der Corput set is intersective. There is a more quantitative version of this fact, due to Ruzsa~\cite{Ruz}: Define
\[ \gamma_{\HH}=\inf\int_0^1 T(x)\,dx\text{\ \ where \ }
T(x)=a_0+\sum_{h\in\HH}a_h\cos(2\pi hx)\ge 0 \text{\ \ and \ }\;T(0)=1,\]
$T$ being a trigonometric polynomial. It it is known that $\HH$ is a  van der Corput set iff $\gamma_{\HH}=0$, see~\cite{Ruz} and~\cite{Mont}. Also, let
\[ \delta_{\HH}=\sup\{D^*(\A)\::\: (\A-\A)\cap \HH=\emptyset\}.\]
By definition, $\HH$ is intersective iff $\delta_{\HH}=0$. It was shown in~\cite{Ruz} that
$\delta_{\HH}\le \gamma_{\HH}$. In view of this fact, our Theorem~\ref{thm:main} has the following consequence.
\begin{cor}
\it Let $\Seq=\{n_j\}_j$ be lacunary with ratio $1+\eps$,
 and suppose $T$ is a nonnegative trigonometric polynomial of the form
\[ T(x)=a_0+\sum_{j}a_j\,\cos(2\pi n_j x)\text{\ \ with\ }T(0)=1.\]
Then for some universal constant $c>0$,
\be\label{postrig} \int_0^1 T(x)\,dx &=& a_0 > c\eps\,|\log\eps|^{-1}.\ee
\end{cor}

%This inequality can also be proved directly using the classical approach of
%Riesz products (see \cite[Sec.\ V.1, Ex.~7]{kbook}).
% On the other hand, 
If one could show that the bound given by~\eqref{postrig} is optimal,
then it would follow that the
$|\log \eps|^{-1} $ factor in~\eqref{mlogm} cannot be removed.

\noindent{\bf Remark.} The first author heard Y. Katznelson present his proof of (\ref{katz})
in a lecture at Stanford in 1991, but that proof only appeared ten years later~\cite{Kat}.
The proof of our main result, Theorem~\ref{thm:main},  was obtained in 1999 and presented 
in a lecture~\cite{lec} at the IAS, Princeton in 2000. 
We thank J. Grytczuk for urging us
to publish this result and providing references for recent work on related problems.

\end{document}